\documentclass[12pt,oneside,american,british,english]{amsart}
\usepackage[T1]{fontenc}
\usepackage[utf8]{inputenc}
\usepackage{textcomp}
\usepackage{amstext}
\usepackage{amsthm}
\usepackage{amssymb}
\usepackage{esint}
\usepackage{color}
\inputencoding{utf8}

\newcommand{\rmd}{\mathrm{d}}

\newcommand{\bt}{\beta}
\newcommand{\ep}{\varepsilon}

\newcommand{\lm}{\lambda}

\newcommand{\om}{\omega}

\newcommand{\cadlag}{c\`adl\`ag }

\newcommand{\mc}[1]{\mathcal{#1}}
\newcommand{\mb}[1]{\mathbb{#1}}
\newcommand{\mbf}[1]{\mathbf{#1}}

\newcommand{\bi}{\begin{itemize}}
\newcommand{\ei}{\end{itemize}}
\newcommand{\bg}{\begin}
\newcommand{\e}{\end}
\newcommand{\ed}{\end{document}}

\newcommand{\be}{\begin{enumerate}}
\newcommand{\ee}{\end{enumerate}}
\newcommand{\beq}{\begin{equation}}
\newcommand{\eeq}{\end{equation}}
\newcommand{\beqs}{\begin{equation*}}
\newcommand{\eeqs}{\end{equation*}}
\newcommand{\beqa}{\begin{eqnarray}}
\newcommand{\eeqa}{\end{eqnarray}}
\newcommand{\beqas}{\begin{eqnarray*}}
\newcommand{\eeqas}{\end{eqnarray*}}

\newcommand{\aP}[1]{
	\ensuremath{\langle #1,#1 \rangle}}

\newcommand{\OFP}{\lk\Omega,\mathcal{F},\mathbf{P}\rk}

\numberwithin{equation}{section}

\newcommand{\lk}{\left(}
\newcommand{\rk}{\right)}

\makeatletter
\numberwithin{equation}{section}
\numberwithin{figure}{section}
\theoremstyle{plain}
\newtheorem{thm}{\protect\theoremname}[section]
  \theoremstyle{plain}
  \newtheorem{assumption}[thm]{\protect\assumptionname}
  \theoremstyle{definition}
  \newtheorem{defn}[thm]{\protect\definitionname}
  \theoremstyle{remark}
  \newtheorem{rem}[thm]{\protect\remarkname}
  \theoremstyle{plain}
  \newtheorem{lem}[thm]{\protect\lemmaname}
  \theoremstyle{plain}
  \newtheorem{cor}[thm]{\protect\corollaryname}

\usepackage[displaymath]{lineno}

\makeatother

\usepackage{babel}
  \addto\captionsamerican{\renewcommand{\assumptionname}{Assumption}}
  \addto\captionsamerican{\renewcommand{\corollaryname}{Corollary}}
  \addto\captionsamerican{\renewcommand{\definitionname}{Definition}}
  \addto\captionsamerican{\renewcommand{\lemmaname}{Lemma}}
  \addto\captionsamerican{\renewcommand{\remarkname}{Remark}}
  \addto\captionsamerican{\renewcommand{\theoremname}{Theorem}}
  \addto\captionsbritish{\renewcommand{\assumptionname}{Assumption}}
  \addto\captionsbritish{\renewcommand{\corollaryname}{Corollary}}
  \addto\captionsbritish{\renewcommand{\definitionname}{Definition}}
  \addto\captionsbritish{\renewcommand{\lemmaname}{Lemma}}
  \addto\captionsbritish{\renewcommand{\remarkname}{Remark}}
  \addto\captionsbritish{\renewcommand{\theoremname}{Theorem}}
  \addto\captionsenglish{\renewcommand{\assumptionname}{Assumption}}
  \addto\captionsenglish{\renewcommand{\corollaryname}{Corollary}}
  \addto\captionsenglish{\renewcommand{\definitionname}{Definition}}
  \addto\captionsenglish{\renewcommand{\lemmaname}{Lemma}}
  \addto\captionsenglish{\renewcommand{\remarkname}{Remark}}
  \addto\captionsenglish{\renewcommand{\theoremname}{Theorem}}
  \providecommand{\assumptionname}{Assumption}
  \providecommand{\corollaryname}{Corollary}
  \providecommand{\definitionname}{Definition}
  \providecommand{\lemmaname}{Lemma}
  \providecommand{\remarkname}{Remark}
\providecommand{\theoremname}{Theorem}

\begin{document}

\title{Unexpected Default in an Information Based Model}

\author{M. L. Bedini, R. Buckdahn, H.-J. Engelbert}

\date{\today}
\begin{abstract}
This paper provides sufficient conditions for the time of bankruptcy
(of a company or a state) for being a totally inaccessible stopping
time and provides the explicit computation of its compensator in a
framework where the flow of market information on the default is \foreignlanguage{british}{modelled}
explicitly with a Brownian bridge between 0 and 0 on a random time
interval.
\end{abstract}


\maketitle

\section{Introduction}

\noindent One of the most important objects in a mathematical model
for credit risk is the time $\tau$ (called \textit{default time})
at which a certain company (or state) bankrupts. \foreignlanguage{american}{Modelling}
the flow of market information concerning a default time is crucial
and in this paper we consider a process, $\beta=\left(\beta_{t},\,t\geq0\right)$,
whose natural filtration $\mathbb{F}^{\beta}$ describes the flow of information available for market agents about the time at which the default occurs. For this reason the process $\beta$ will be called the \textit{information process}. In the present paper, we define $\bt$ to be a Brownian bridge between 0 and 0 of random length $\tau$:
\[
\beta_{t}:=W_{t}-\frac{t}{\tau\vee t}W_{\tau\vee t},\quad t\geq0,
\]
where $W=\left(W_{t},\,t\geq0\right)$ is a Brownian motion independent
of $\tau$. 

In this paper, the focus is on the classification of the default time
with respect to the filtration $\mathbb{F}^{\beta}$ and our main
result is the following: If the distribution of the default time $\tau$
admits a continuous density $f$ with respect to the Lebesgue measure, then $\tau$
is a totally inaccessible stopping time and its compensator $K=\left(K_{t},\,t\geq0\right)$
is given by 
\[
K_{t}=\int_{0}^{t\wedge\tau}\frac{f\left(s\right)}{\int_s^\infty v^{\frac{1}{2}}\,(2\pi s\,\left(v-s\right))^{-\frac{1}{2}}\,f\left(v\right)\rmd v}\,\rmd L^{\bt}\left(s,0\right)
\]
where $L^\bt(t,0)$ is the local time of the information process $\bt$ at level $0$ up to time $t$. 

Knowing whether the default time is a predictable, accessible or totally inaccessible stopping time is very important in a mathematical credit risk model. A predictable default time is typical of structural credit risk models, while totally inaccessible default times are one of the most important features of reduced-form credit risk models. In the first framework, market agents know when the default is about to occur, while in the latter default occurs by surprise. The fact that financial
markets cannot foresee the time of default of a company makes the
reduced-form models well accepted by practitioners. In this sense, totally inaccessible default times seem to be the best candidates for modelling times of bankruptcy. We refer, among others, to the papers of Jarrow and Protter \cite{key-10} and of Giesecke \cite{key-9} on the relations between financial information and the properties of the default time, and also to the series of papers of Jeanblanc and Le Cam \cite{key-11,key-12,key-13}. It is remarkable that in our setting the default time is a totally inaccessible stopping time under the common assumption that it admits a continuous density with respect to the Lebesgue measure. Both the hypothesis that the default time admits a continuous density and its consequence that the default occurs by surprise are standard in mathematical credit risk models, but in the information-based approach there is the additional feature of an explicit model for the flow of information which is more sophisticated than the standard approach. There, the available information on the default is modelled through $\left(\mathbb{I}_{\left\{ \tau\leq t\right\} },\,t\geq0\right)$, the single-jump process occurring at $\tau$, meaning that people just know if the default has occurred or not. Financial reality can be more complex and there are actually periods in which default is more likely to happen than in others. In the information based approach, periods of fear of an imminent default correspond to situations where the information process is close to 0, while periods when investors are relatively sure that the default is not going to occur immediately correspond to situations where $\beta_{t}$ is far from 0.

The paper is organized as follows. In Section \ref{sec:The-Information-Process}
we recall the definition and the main properties of the information
process. In Section \ref{sec:The-Compensator-of} we state and prove
Theorem \ref{thm:MAIN} which is the main result of the paper. In Appendix \ref{sec:The-Local-Time} we provide the properties of the
local time associated with the information process. In Appendix \ref{sec:Auxiliary-Results}
we give the proofs of some auxiliary lemmas. Finally, in Appendix \ref{sec:The-Meyer-Approach}, for the sake of easy reference, we recall the so-called Laplacian approach developed by P.-A. Meyer (see, e.g., his book \cite{key-17}) for
computing the compensator of a right-continuous potential of class (D). It is an important ingredient of the approach adopted in this note to determine the compensator of the $\mathbb{F}^{\beta}$-submartingale $\left(\mathbb{I}_{\left\{ \tau\leq t\right\} },\,t\geq0\right)$.

The idea of modelling the information about the default time with
a Brownian bridge defined on a stochastic interval was introduced
in the thesis \cite{key-2}. The definition of the information process
$\beta$, the study of its basic properties and an application to
the problem of pricing a Credit Default Swap (one of the most traded
derivatives in the credit market) have also appeared recently in the
paper \cite{key-3}. 

Non-trivial and sufficient conditions for making the default time
a predictable stopping time will be considered in another paper, \cite{key-31}.
Other topics related with Brownian bridges on stochastic intervals (which will not be considered in this note) are concerned with the problem of studying the progressive enlargement of a reference filtration $\mathbb{F}$ by the filtration $\mathbb{F}^{\beta}$ generated by the information process and further applications to Mathematical Finance.

\section{\label{sec:The-Information-Process} The Information Process and its Basic Properties}

\noindent We start by recalling the definition and the basic properties of a Brownian bridge between $0$ and $0$ of random length. The material of this section resumes some of the results obtained in the paper \cite{key-3}, where we shall refer to for the proofs and more details on the basic properties of such process. 

If $A\subseteq\mathbb{R}$ (where $\mb R$ denotes the set of real numbers), then the set $A_{+}$ is defined as $A_{+}:=A\cap\{x\in\mathbb{R}:x\geq0\}$.
If $E$ is a topological space, then $\mathcal{B}(E)$ denotes the
Borel $\sigma$-algebra over $E$. The indicator function of a set
$A$ will be denoted by $\mathbb{I}_{A}$.  A function $f:\mathbb{R}\rightarrow\mathbb{R}$ will be said to be \textit{\cadlag} if it is right-continuous with
limits from the left.

Let $\left(\Omega,\mathcal{F},\mathbf{P}\right)$ be
a complete probability space. We denote by $\mathcal{N}_{P}$ the collection of $\mathbf{P}$-null sets of $\mathcal{F}$. If $\mathcal{L}$ is the law of the random variable $\xi$ we shall write $\xi\sim\mathcal{L}$. Unless otherwise specified, all filtrations considered in the following are supposed to satisfy the usual conditions of right continuity and completeness. 

Let $\tau:\Omega\rightarrow\left(0,+\infty\right)$
be a strictly positive random time, whose distribution function is
denoted by $F$: $F\left(t\right):=\mathbf{P}\left(\tau\leq t\right),\;t\in\mathbb{R}_{+}$.
The time $\tau$ models the random time at which some default occurs and hereinafter it will be called \textit{default time}.

Let $W=\left(W_{t},\,t\geq0\right)$ be a Brownian motion defined on $\OFP$ and starting from 0. We shall always make use of the following assumption:
\begin{assumption}
\label{ass:Assumption} The random time $\tau$ and the Brownian motion
$W$ are independent.
\end{assumption}
Given $W$ and a strictly positive real number $r$, a standard Brownian bridge $\beta^{r}=\left(\beta_{t}^{r},\,t\geq 0\right)$ between $0$ and $0$ of length $r$ is defined by 
$$
\beta^{r}_t=W_{t}-\frac{t}{r\vee t}W_{r\vee t},\quad t\geq0\,.
$$
For further references on Brownian bridges, see, e.g., Section 5.6.B
of the book \cite{key-16} of Karatzas and Shreve. 

Now we are going to introduce the definition of the Brownian bridge of random length (see \cite{key-3}, Definition 3.1).
\begin{defn}{\em
The process $\beta=\left(\beta_{t},\,t\geq0\right)$ given by
\begin{equation}
\beta_{t}:=W_{t}-\frac{t}{\tau\vee t}W_{\tau\vee t},\quad t\geq0\,,\label{eq:info-proc}
\end{equation}
will be called Brownian bridge of random length $\tau$. We will often say that $\beta=\left(\beta_{t},\,t\geq0\right)$ is the information process (for the random time $\tau$ based on $W$).}
\end{defn} 
The natural filtration of $\bt$ will be denoted by $\mathbb{F}^{\beta}=(\mathcal{F}_{t}^{\beta})_{t\geq0}$:
\[
\mathcal{F}_{t}^{\beta}:=\sigma\left(\beta_{s},\,0\leq s\leq t\right)\vee\mathcal{N}_{P}\,.
\]
Note that according to \cite{key-3}, Corollary 6.1, the filtration $\mb F^\beta$ (denoted therein by $\mb F^P$) satisfies the usual conditions of right-continuity and completeness.
\begin{rem}
The law of $\beta$, conditional on $\tau=r$, is
the same as that of a standard Brownian bridge between $0$ and $0$ of length $r$ (see \cite{key-3}, Lemma 2.4 and Corollary 2.2). In particular, if $0<t<r$, the law of $\bt_t$, conditional on $\tau=r$, is Gaussian with expectation zero and variance $\frac{t\left(r-t\right)}{r}$:
\[
\mbf P\left(\bt_t\in\cdot\;\big|\tau=r\right)=\mathcal{N}\left(0,\frac{t\left(r-t\right)}{r}\right),
\]
where $\mathcal{N}\left(\mu,\sigma^{2}\right)$ denotes the Gaussian
law of mean $\mu$ and variance $\sigma^{2}$.
\end{rem} 

By $p\left(t,\cdot,y\right)$ we denote the density of a Gaussian random variable with mean $y\in\mathbb{R}$ and variance $t>0$:
\begin{equation}
p\left(t,x,y\right):=\frac{1}{\sqrt{2\pi t}}\exp\left[-\frac{\left(x-y\right)^{2}}{2t}\right],\,x\in\mathbb{R}.\label{eq:GaussDens}
\end{equation}
For later use we also introduce the functions $\varphi_t$ ($t>0$): 
\begin{equation}\label{eq:cond-dens}
\varphi_{t}\left(r,x\right):=\begin{cases}
p\left(\frac{t\left(r-t\right)}{r},x,0\right), & 0<t<r,\ x\in\mathbb{R},\\
0, & r\leq t,\ x\in\mathbb{R}.
\end{cases}
\end{equation}
We notice that for $0<t<r$ the conditional density of $\beta_{t}$, conditional on $\tau=r$, is just
equal to the density $\varphi_{t}\left(r,\cdot\right)$ of a standard
Brownian bridge $\beta^{r}_t$ of length $r$ at time $t$.

We proceed with the property that the default time $\tau$ is nonanticipating with respect to the filtration $\mb F^\bt$ and the Markov property of the information process $\bt$. 

\begin{lem}
\label{lem:For-all-,}For all $t>0$, $\left\{ \beta_{t}=0\right\} =\left\{ \tau\leq t\right\} ,\:\mathbf{P}$-a.s.
In particular, $\tau$ is an $\mathbb{F}^{\beta}$-stopping time.\end{lem}
\begin{proof}
See \cite{key-3}, Proposition 3.1 and Corollary 3.1.
\end{proof}
\begin{thm}\label{MP}
The information process $\beta$ is a Markov process with respect to the filtration $\mb F^{\beta}$: For all $0\le t<u$ and measurable real functions $g$ such that $g(\bt_u)$ is integrable,
$$
\mbf E[g(\bt_u)|\mc F^\bt_t]=\mbf E[g(\bt_u)|\bt_t],\quad \mbf P\mbox{-a.s.}
$$
\end{thm}
\vspace{-10pt}
\begin{proof}
See Theorem 6.1 in \cite{key-3}.
\end{proof}
As the following theorem together with Theorem \ref{MP} shows, the function $\phi_{t}$ defined by 
\begin{equation}
\phi_{t}\left(r,x\right):=\frac{\varphi_{t}\left(r,x\right)}{{\displaystyle \int_{\left(t,+\infty\right)}\varphi_{t}\left(v,x\right)\,\rmd F\left(v\right)}},\label{eq:densitybeta2}
\end{equation}
$\left(r,t\right)\in\left(0,+\infty\right)\times
\mathbb{R}_{+}$, $x\in\mathbb{R}$, is, for $t<r$, the a posteriori density function of $\tau$ on $\{\tau>t\}$, conditional on $\beta_{t}=x$. 
\begin{thm}
\label{thm:COND EXP}Let $t>0$, $g:\mathbb{R}^{+}\rightarrow\mathbb{R}$
be a Borel function such that $\mathbf{E}\left[\left|g\left(\tau\right)\right|\right]<+\infty$.
Then, $\mathbf{P}$-a.s. 
\begin{align}
\mathbf{E}\left[g\left(\tau\right)|\mathcal{F}_{t}^{\beta}\right]=g\left(\tau\right)\mathbb{I}_{\left\{ \tau\leq t\right\} }+{\displaystyle \int_{\left(t,+\infty\right)}g\left(r\right)}\,\phi_{t}\left(r,\beta_{t}\right)\,\rmd F\left(r\right)\mathbb{I}_{\left\{ t<\tau\right\} }.\label{eq:condexptau-1}
\end{align}
\end{thm}
\begin{proof}
See Theorem 4.1, Corollary 4.1 and Corollary 6.1 in \cite{key-3}.
\end{proof}
Before stating the next result which is concerned with the semimartingale decomposition
of the information process, let us give the following definition:
\begin{defn}\label{stoppedBM}
{\em
Let $B$ be a continuous process, $\mathbb{F}$ a filtration and $T$
an $\mathbb{F}$-stopping time. Then $B$ is called an $\mathbb{F}$-Brownian motion stopped at $T$ if $B$ is an $\mathbb{F}$-martingale with
square variation process $\langle B,B\rangle_{t}=t\wedge T$, $t\geq0$.}
\end{defn}
Now we introduce the real-valued function $u$ defined by
\begin{equation}
u\left(s,x\right):=\mathbf{E}\Big[\frac{\beta_{s}}{\tau-s}\mathbb{I}_{\left\{ s<\tau\right\} }\big|\beta_{s}=x\Big],\quad s\in\mathbb{R}_{+},\;x\in\mathbb{R}. \label{eq:u}
\end{equation}
\begin{thm}
\label{thm:beta semimartingaale} The process $b$ defined by 
$$
b_{t}:=\beta_{t}+\int_{0}^{t}u(s,\bt_s)\,\rmd s,\quad t\geq0\,,
$$
is an $\mathbb{F}^{\beta}$-Brownian motion stopped at $\tau$. The
information process $\beta$ is therefore an $\mathbb{F}^{\beta}$-semimartingale
with decomposition
\begin{equation}\label{eq:semimart-dec-1}
\beta_{t}=b_{t}-\int_{0}^{t\wedge\tau}u\left(s,\beta_{s}\right)\,\rmd s,\quad t\geq 0\,.
\end{equation}
\end{thm}
\begin{proof}
See Theorem 7.1 in \cite{key-3}.
\end{proof}
\begin{rem}
The quadratic variation of the information process $\bt$ is given by
\begin{equation}\label{eq: beta beta t}
\left\langle \beta,\beta\right\rangle _{t}=\left\langle b,b\right\rangle _{t}=t\wedge\tau, \quad t\geq 0\,.
\end{equation}
\end{rem}

\section{\label{sec:The-Compensator-of}The Compensator of the Default Time}

\noindent In this section we compute explicitly the compensator of
the single-jump process with jump occurring at $\tau$, which will be denoted
by $H$=$\left(H_{t},\,t\geq0\right)$: 
\begin{equation}
H_{t}:=\mathbb{I}_{\left\{ \tau\leq t\right\} },\quad t\geq0.\label{eq:H}
\end{equation}
The process $H$, called default process, is an $\mathbb{F}^{\beta}$-submartingale and its
compensator is also known as the compensator of the
$\mathbb{F}^{\beta}$-stopping time $\tau$. Our main goal consists in providing a representation of the compensator of $H$. As we shall see below, this representation involves the local time $L^\bt(t,0)$ of the information process $\bt$ (see Appendix A for properties of local times of continuous semimartingales and, in particular, of $\bt$). From its representation we immediately obtain that the compensator of the default process $H$ is continuous. As a result, from the continuity
of the compensator of $H$ it follows that the default time $\tau$ is a totally inaccessible $\mb F^\bt$-stopping time. 

In this section the following assumption will always be in force.
\begin{assumption}
\label{assu:The-distribution-function}
{\rm(i)} The distribution function $F$ of $\tau$ admits a continuous density function $f$ with respect to the Lebesgue measure $\lambda_+$ on $\mb R_+$.

{\rm(ii)} $F(t)<1$ for all $t\geq 0$.
\end{assumption}

The following theorem is the main result of this paper:
\begin{thm}\label{thm:MAIN}
Suppose that Assumption \ref{assu:The-distribution-function}
is satisfied.\\
\hspace*{8pt}{\rm (i)} \ The process $K=\left(K_{t},\,t\geq0\right)$
defined by 
\begin{equation}
K_{t}:=\int_{0}^{t\wedge\tau}\frac{f\left(s\right)}{\int_s^\infty\varphi_{s}\left(v,0\right)f\left(v\right)dv}\,\rmd L^\bt(s,0),\quad t\geq 0\,,\label{eq:K}
\end{equation}
is the compensator of the default process $H$. Here $L^\bt(t,x)$ denotes the local time of the information process $\bt$ up to $t$ at level $x$.\\
\hspace*{8pt}{\rm (ii)} \ The default time $\tau$ is a totally inaccessible stopping time with respect to the filtration $\mathbb{F}^{\beta}$.
 \end{thm}
\bg{proof} First we verify statement (ii) under the supposition that (i) is true. Obviously, as $L^\bt(s,0)$ is continuous in $s$ (see Lemma \ref{lem:local time x continuity}), the process $K$ given by \eqref{eq:K} is continuous. Consequently, because of the well-known equivalence between this latter property and the continuity of the compensator (see, e.g., \cite{key-15}, Corollary 25.18), we can conclude that the default time $\tau$ is a totally inaccessible stopping time with respect to $\mb F^\bt$.

Now we prove statement (i) of the theorem. For every $h>0$ we define the process $K^{h}=\left(K_{t}^{h},\,t\geq0\right)$
by
\begin{eqnarray}
K_{t}^{h}&:=&\frac{1}{h}\int_{0}^{t}\left(\mathbb{I}_{\left\{ s<\tau\right\} }-\mathbf{E}\left[\mathbb{I}_{\left\{ s+h<\tau\right\} }|\mathcal{F}_{s}^{\beta}\right]\right)\,\rmd s\nonumber\\
&=&\int_{0}^{t}\frac{1}{h}\mathbf{P}\left(s<\tau<s+h|\mathcal{F}_{s}^{\beta}\right)\,\rmd s\,,\quad\mathbf{P}\textrm{-a.s.}\label{eq:KAPPAacca-1}
\end{eqnarray}

The proof is divided into two parts. In the first part we prove that $K_t-K_{t_0}$ is the $\mathbf{P}$-a.s. limit of $K_t^{h}-K^h_{t_0}$ as $h\downarrow0$, for every $t_0,t\geq0$ such that $0<t_0<t$. In the second part of the proof we show that the process $K$ is indistinguishable from the compensator of $H$. Auxiliary results used throughout the proof are postponed to Appendix \ref{sec:Auxiliary-Results}. 

For the first part of the proof, we fix $t_0,t$ such that $0<t_0<t$ and notice that
\begin{align}
K_{t}^{h}-K_{t_0}^{h} & =\int_{t_0}^{t}\frac{1}{h}\mathbf{P}\left(s<\tau<s+h|\mathcal{F}_{s}^{\beta}\right)\,\rmd s\nonumber\\
 & =\int_{t_0\wedge\tau}^{t\wedge\tau}\frac{1}{h}\left(\frac{\int_{s}^{s+h}\varphi_{s}\left(r,\beta_{s}\right)f\left(r\right)\,\rmd r}{\int_s^\infty\varphi_{s}\left(v,\beta_{s}\right)f\left(v\right)\,\rmd v}\right)\,\rmd s\label{limit_K}
\end{align}
where the last equality is a consequence of Theorem \ref{thm:COND EXP} and Definition (\ref{eq:densitybeta2}) of the a posteriori density function of $\tau$. Later we shall verify that 
\beq\label{limit_Kf}
\lim_{h\downarrow 0} \int_{t_0\wedge\tau}^{t\wedge\tau}\frac{1}{h}\left(\frac{\int_{s}^{s+h}\varphi_{s}\left(r,\beta_{s}\right)\left[f\left(r\right)-f\left(s\right)\right]\,\rmd r}{\int_s^\infty\varphi_{s}\left(v,\beta_{s}\right)f\left(v\right)\,\rmd v}\right)\,\rmd s=0\quad \mbf P\mbox{-a.s.}
\eeq
So, we have to deal with the limit behaviour as $h\downarrow 0$ of 
\beqa
\nonumber\lefteqn{\int_{t_0\wedge\tau}^{t\wedge\tau}\frac{1}{h}\left(\frac{\int_{s}^{s+h}\varphi_{s}\left(r,\beta_{s}\right)\,\rmd r}{\int_s^\infty\varphi_{s}\left(v,\beta_{s}\right)f\left(v\right)\,\rmd v}\right)\,f\left(s\right)\,\rmd s}\\
\nonumber&=&\int_{t_0\wedge\tau}^{t\wedge\tau}\frac{1}{h}\left(\frac{\int_{0}^{h}\varphi_{s}\left(s+u,\beta_{s}\right)\,\rmd u}{\int_s^\infty\varphi_{s}\left(v,\beta_{s}\right)f\left(v\right)\,\rmd v}\right)\,f\left(s\right)\rmd s\\
\label{limit_K-1}&=&\int_{t_0\wedge\tau}^{t\wedge\tau}\frac{1}{h}\int_{0}^{h}p\left(\frac{su}{s+u},\beta_{s},0\right)\,\rmd u\, g(s,\bt_s)\,f\left(s\right)\,\rmd s
\eeqa
where we have introduced the function $g:\left(0,+\infty\right)\times\mathbb{R}\rightarrow\mathbb{R}_{+}$ by
\begin{equation}
g\left(s,x\right):=\Big(\int_s^\infty\varphi_{s}\left(v,x\right)f\left(v\right)\,\rmd v\Big)^{-1},\quad s>0,\,x\in\mathbb{R}\,.\label{eq:g}
\end{equation}
In \eqref{limit_K-1}, we want to replace $p\left(\frac{su}{s+u},\beta_{s},0\right)$ by $p(u,\bt_s,0)$. To this end, we estimate
\beqa
\nonumber\lefteqn{\Big\vert p\Big(\frac{su}{s+u},x,0\Big)-p(u,x,0)\Big\vert}\\
\nonumber\hspace*{20pt}&=&p(u,x,0)\,\Big\vert\left(\frac{s+u}{s}\right)^\frac{1}{2} \exp\Big(-\frac{x^2}{2s}\Big)-1\Big\vert\\
\nonumber\hspace*{20pt}&\le&p(u,x,0)\Bigg[\left(\frac{s+u}{s}\right)^\frac{1}{2} \Big\vert\exp\Big(-\frac{x^2}{2s}\Big)-1\Big\vert
+\Big\vert\left(\frac{s+u}{s}\right)^\frac{1}{2} -1\Big\vert\Bigg]\\
\nonumber\hspace*{20pt}&\le&\left((2\pi)^\frac{1}{2}|x|\right)^{-1}\exp\left(-\frac{1}{2}\right)\left(\frac{s+1}{s}\right)^\frac{1}{2}\Big( \frac{x^2}{2s}\Big)+(2\pi u)^{-\frac{1}{2}} \left(\frac{u}{2s}\right)\\
\hspace*{20pt}&\le&c_1|x|+c_2\,u^\frac{1}{2}\,,\label{Ineq1}
\eeqa
with some constants $c_1$ and $c_2$, for $0\le u\le h\le1$ and $s\in[t_0,t]$, where for the estimate of the first summand we have used that the function $u\mapsto p(u,x,0)$ has its unique maximum at $u=x^2$, the standard estimate $1-e^{-z}\leq z$ for all $z\geq 0$ and that $(s+1)s^{-1}\le 1+t_0^{-1}$ as well as for the estimate of the second summand the inequalities $p(u,x,0)\le(2\pi u)^{-\frac{1}{2}}$ and $|\sqrt{\frac{s+u}{s}}-1|\le \frac{u}{2s}$. Putting $x=\bt_s$ and integrating from $0$ to $h$ and dividing by $h$, for $0\le u\le h\le1$ and $s\in[t_0,t]$, from \eqref{Ineq1} we obtain
\beqas
\lefteqn{\frac{1}{h}\int_{0}^{h}\Big\vert p\big(\frac{su}{s+u},\beta_{s},0\big) -p\left(u,\beta_{s},0\right)\Big\vert\,\rmd u\,  g(s,\bt_s)\,f\left(s\right)}\\
&\le&\big(c_1|\bt_s|+c_2\,h^\frac{1}{2}\big)\,g(s,\bt_s)\,f\left(s\right)\\
&\le&\big(c_1|\bt_s|+c_2\big)\,c_3\,C(t_0,t,\bt_s)
\eeqas
where $C(t_0,t,x)$ is an upper bound of $g(s,x)$ on $[t_0,t]$ continuous in $x$ (see Lemma \ref{lem:Lemma3}) and $c_3$ is an upper bound for the continuous density function $f$ on $[t_0,t]$. The right hand side is integrable over $[t_0,t]$ with respect to the Lebesgue measure $\lm_+$. On the other side, by the fundamental theorem of calculus we have that, for every $x\not=0$, 
\beq\label{q-density}
\lim_{h\downarrow 0}\frac{1}{h}\int_0^h p(u,x,0)\,\rmd u=0,\quad
\lim_{h\downarrow 0}\frac{1}{h}\int_0^h p\Big(\frac{su}{s+u},x,0\Big)\,\rmd u=0\,.
\eeq
For this we notice that $p(0,x,0)=0$ is a continuous extension of the function $u\mapsto p(u,x,0)$ if $x\not=0$. By Corollary \ref{lem:occupation-time2}, we have that the set $\{0\le s\le t\wedge \tau: \ \bt_s=0\}$ has Lebesgue measure zero. Then using Lebesgue's theorem on dominated convergence, we can conclude that $\mbf P$-a.s.
\beq\label{Comparison}
\lim_{h\downarrow 0}\int_{t_0\wedge\tau}^{t\wedge\tau}\frac{1}{h}\int_{0}^{h}\Big\vert p\big(\frac{su}{s+u},\beta_{s},0\big)-p(u,\bt_s,0)\Big\vert\,\rmd u\, g(s,\bt_s)\,f\left(s\right)\,\rmd s=0\,.
\eeq
This finishes the first step of the proof of the first part saying that in \eqref{limit_K-1} we may replace $p\left(\frac{su}{s+u},\beta_{s},0\right)$ by $p(u,\bt_s,0)$ for identifying the limit. 

The second step of the first part is to prove that 
\beqas
\lim_{h\downarrow 0} \int_{t_0\wedge\tau}^{t\wedge\tau}\frac{1}{h}\int_{0}^{h}p\left(u,\beta_{s},0\right)\,\rmd u\, g(s,\bt_s)\,f\left(s\right)\,\rmd s=K_t-K_{t_0}\,,\quad \mbf P\mbox{-a.s.}
\eeqas
Setting 
\beq\label{density-q}
q(h,x):=\frac{1}{h}\int_0^h p(u,x,0)\,\rmd u,\quad 0<h\le 1, \ x\in\mb R\,,
\eeq
an application of the occupation time formula (see Corollary \ref{lem:occupation-time2}) yields
\beqa\label{AppOTF}
\nonumber\lefteqn{\int_{t_0\wedge\tau}^{t\wedge\tau}\frac{1}{h}\int_{0}^{h}p\left(u,\beta_{s},0\right)\,\rmd u\, g(s,\bt_s)\,f\left(s\right)\,\rmd s}\\
\nonumber&=&\int_{t_0\wedge\tau}^{t\wedge\tau}\,q(h,\bt_s)\, g(s,\bt_s)\,f\left(s\right)\,\rmd s\\
&=&\int_{-\infty}^{+\infty}\left(\int_{t_0}^{t}g\left(s,x\right)f\left(s\right)\,\rmd L^\bt(s,x)\right)q\left(h,x\right)\,\rmd x\,,\quad
\mbf P\mbox{-a.s.}
\eeqa
For every $h>0$, $q(h,\cdot)$ is a probability density function with respect to the Lebesgue measure on $\mb R$. According to Lemma \ref{q-weak convergence}, the probability measures $\mbf Q_h$ with density $q(h,\cdot)$ converge weakly to the Dirac measure $\delta_0$ at $0$. On the other hand, Lemma \ref{lem:Lemma4} shows that the function $x\mapsto \int_{t_0}^{t}g\left(s,x\right)f\left(s\right)\,\rmd L^\bt(s,x)$ is continuous and bounded. Hence, in \eqref{AppOTF} we can pass to the limit and obtain
\beqa\label{Result-Step-2}
\nonumber\lefteqn{\lim_{h\downarrow0}\int_{t_0\wedge\tau}^{t\wedge\tau}\frac{1}{h}\int_{0}^{h}p\left(u,\beta_{s},0\right)\,\rmd u\, g(s,\bt_s)\,f\left(s\right)\,\rmd s}\\
&=&\int_{t_0}^{t}g\left(s,0\right)f\left(s\right)\,\rmd L^\bt(s,0),\quad
\mbf P\mbox{-a.s.}
\eeqa

In the third step of the proof of the first part we have to show that \eqref{limit_Kf} holds. Note that the function $f$ is uniformly continuous on $[t_0,t+1]$. We fix $\ep>0$ and choose $0<\delta\le1$ such that $|f(s+u)-f(s)|\le\ep$ for every $0\le u<\delta$. Proceeding similarly as above, we obtain
\beqas
\lefteqn{\limsup_{h\downarrow 0} \Big\vert\int_{t_0\wedge\tau}^{t\wedge\tau}\frac{1}{h}\left(\frac{\int_{s}^{s+h}\varphi_{s}\left(r,\beta_{s}\right)\left[f\left(r\right)-f\left(s\right)\right]\,\rmd r}{\int_s^\infty\varphi_{s}\left(v,\beta_{s}\right)f\left(v\right)\,\rmd v}\right)\,\rmd s\Big\vert}\\
&\hspace*{-2pt}\le&\hspace{-6pt}\limsup_{h\downarrow 0}\int_{t_0\wedge\tau}^{t\wedge\tau}\frac{1}{h}\int_{0}^{h}p\left(\frac{su}{s+u},\beta_{s},0\right)\,\big|f(s+u)-f(s)\big|\,\rmd u\, g(s,\bt_s)\,\rmd s\\
&\hspace*{-2pt}\le&\hspace{-6pt}\ep\,\limsup_{h\downarrow 0}\int_{t_0\wedge\tau}^{t\wedge\tau}\frac{1}{h}\int_{0}^{h}p\left(\frac{su}{s+u},\beta_{s},0\right)\,\rmd u\, g(s,\bt_s)\,\rmd s\\
&\hspace*{-2pt}=&\hspace{-6pt}\ep\,\limsup_{h\downarrow 0}\int_{t_0\wedge\tau}^{t\wedge\tau}\frac{1}{h}\int_{0}^{h}p\left(u,\beta_{s},0\right)\,\rmd u\, g(s,\bt_s)\,\rmd s\\
&\hspace*{-2pt}=&\hspace{-6pt}\ep\,\int_{t_0}^{t}g\left(s,0\right)\,\rmd L^\bt(s,0),\quad
\mbf P\mbox{-a.s.}
\eeqas
Since $\ep>0$ is choosen arbitrarily and the integral above is $\mbf P$-a.s. finite, we can conclude that \eqref{limit_Kf} holds. 

The first part of the proof is completed.

The second part of the proof relies on the so-called Laplacian approach
of P.-A. Meyer and, for the sake of easy reference, related results
are recalled in Appendix \ref{sec:The-Meyer-Approach}. Let us denote by $K^{w}$ the compensator of the default process $H$ introduced in \eqref{eq:H}: $H_{t}:=\mathbb{I}_{\left\{ \tau\leq t\right\}},\ t\geq0$. We first show that $K^{h}_t$ converges to $K^{w}_t$ as $h\downarrow0$ in the sense of the weak topology $\sigma\left(L^{1},L^{\infty}\right)$ (see Definition \ref{def:weak L1 Linf conv}), for every $t\geq0$. We then prove that the process $K$ is actually indistinguishable from $K^{w}$.

For the sake of simplicity of the notation, if a sequence of integrable random variables $\left(\xi_{n}\right)_{n\in\mathbb{N}}$
converges to an integrable random variable $\xi$ in the
sense of the weak topology $\sigma\left(L^{1},L^{\infty}\right)$
we will write
\[
\xi_{n}\xrightarrow[n\rightarrow+\infty]{^{\sigma\left(L^{1},L^{\infty}\right)}}\xi.
\]
Furthermore, we will denote by $G$ the right-continuous potential of class (D) (cf. beginning of Appendix \ref{sec:The-Meyer-Approach}) given by 
\begin{equation}
G_{t}:=1-H_{t}=\mathbb{I}_{\left\{ t<\tau\right\}},\quad t\geq0\,.\label{eq:G}
\end{equation}

By Corollary \ref{thm:Kh Kw}, we know that there exists a unique integrable predictable increasing process $K^{w}=\left(K_{t}^{w},\,t\geq0\right)$
which generates, in the sense of Definition \ref{def:potential generated by A}, the potential $G$ given by (\ref{eq:G}) and, for every
$\mathbb{F}^{\beta}$-stopping time $T$, we have that 
\[
K_{T}^{h}\xrightarrow[h\downarrow0]{^{\sigma\left(L^{1},L^{\infty}\right)}}K_{T}^{w}.
\]
The process $K^{w}$ is actually the compensator of $H$. Indeed,
it is a well known fact that the process $H$ admits a unique decomposition
\begin{equation}
H=M+A\label{eq:uno-due-1}
\end{equation}
into the sum of a right-continuous martingale $M$ and an adapted,
natural, increasing, integrable process $A$. The process $A$ is
then called the compensator of $H$. On the other hand, from the definition
of the potential generated by an increasing process (see Definition \ref{def:potential generated by A})
the process 
\begin{equation}
L:=G+K^{w}\label{eq:uno-tre-1}
\end{equation}
is a martingale. By combining the definition (\ref{eq:G}) of the
process $G$ and (\ref{eq:uno-tre-1}) we obtain the following decomposition
of $H$:
\[
H=1-L+K^{w}.
\]
However, by the uniqueness of the decomposition (\ref{eq:uno-due-1}),
we can identify the martingale $M$ with $1-L$ and we have that
$A=K^{w}$,
up to indistinguishability. Since the submartingale $H$ and the martingale
$1-L$ appearing in the above proof are right-continuous, the process
$K^{w}$ is right-continuous, too.

By applying Lemma \ref{cor:Let--be-1} we see that $K_t-K_{t_0}$
is a modification of $K^{w}_t-K^w_{t_0}$, for all $t_0,t$ such that $0<t_0<t$. Passing to the limit as $t_0\downarrow 0$, we get $K_t=K_t^w$ $\mbf P$-a.s. for all $t\geq 0$. Since both processes have right-continuous sample paths they are indistinguishable. 

The theorem is proved.
\end{proof}
\begin{rem}
We close this part of the present paper with the following observations.

(1)\ Note that $\left(\mathbb{I}_{\left\{ \tau\leq t\right\} },\,t\geq0\right)$
does not admit an intensity with respect to the filtration $\mathbb{F}^{\beta}$
(hence it is not possible to apply, for example, Aven's Lemma
for computing the compensator (see, e.g. \cite{key-1}).

(2) Assumption \ref{assu:The-distribution-function}(ii) on the distribution function $F$ that $F(t)<1$ for all $t\geq0$ ensures that the denominator of the integrand of the right-hand side of \eqref{eq:K} is always strictly positive. However, it can be removed. Indeed, if the density function $f$ of $\tau$ is continuous (as required by Assumption \ref{assu:The-distribution-function}(i)), then exactly as above we can show that relation \eqref{eq:K} is satisfied for all $t\le t_1:=\sup\{t>0: \ F(t)<1\}$. On the other hand, it is obvious that $\tau\leq t_1$ $\mathbf P$-a.s. (hence the right-hand side of \eqref{eq:K} is constant for $t\in[t_1,\infty)$) and also that the compensator $K=\left(K_{t},\,t\geq0\right)$
of $\left(\mathbb{I}_{\left\{ \tau\leq t\right\} },\,t\geq0\right)$
is constant on $[t_1,\infty)$. Altogether, it follows that relation \eqref{eq:K} is satisfied for all $t\geq 0$.

\end{rem}
\begin{appendix}
\section{\label{sec:The-Local-Time} On the Local Time of the Information Process}
\noindent In this section we introduce and study the local time process associated with the information process.

For any continuous semimartingale $X=\left(X_{t},\,t\geq0\right)$
and for any real number $x$, it is possible to define the (right) local time
$L^X(t,x)$ associated with $X$ at level $x$ up to time $t$ through Tanaka's formula (see, e.g., \cite{key-18}, Theorem VI.(1.2)) as
follows: 
\begin{equation}
L^X(t,x):=|X_{t}-x|-|X_{0}-x|-\int_{0}^{t}\textrm{sign}\left(X_{s}-x\right)\,\rmd X_{s},\quad t\geq0,\label{eq:tanaka}
\end{equation}
where $\textrm{sign}\left(x\right):=1$ if $x>0$ and $\textrm{sign}\left(x\right):=-1$
if $x\leq0$. The process $L^X\left(\cdot,x\right)=\left(L^X\left(t,x\right),\,t\geq0\right)$
appearing in relation (\ref{eq:tanaka}) is called the (right) \textit{local time of $X$ at level $x$}. 

Now we recall the occupation time formula for local times of continuous semimartingales which is given in a form convenient for our applications. By $\aP{X}$ we denote the square variation process of a continuous semimartingale $X$.
\begin{lem}
\label{lem:OccTimeFormula}Let $X=\left(X_{t},\,t\geq0\right)$ be
a continuous semimartingale. There is a $\mathbf{P}$-negligible set
outside of which
\[
\int_{0}^{t}h\left(s,X_{s}\right)d\left\langle X,X\right\rangle _{s}=\int_{-\infty}^{+\infty}\left(\int_{0}^{t}h\left(s,x\right)\,\rmd L^X\left(s,x\right)\right)\,\rmd x\,,
\]
for every $t\geq0$ and every non-negative Borel function $h$ on $\mathbb{R}_{+}\times\mathbb{R}$.
\end{lem}
\begin{proof}
See Corollary VI.(1.6) of the book of Revuz and Yor \cite{key-18} for the case when $h$ is a non-negative Borel function defined on $\mathbb{R}$ (i.e., it does not depend on time). The statement of the lemma is then proved by first considering the case in which $h$ has the form
$
h\left(t,x\right)=\mathbb{I}_{\left[u,v\right]}\left(t\right)\gamma\left(x\right)
$
for $0\le u<v<\infty$ and a non-negative Borel function $\gamma$ on $\mathbb{R}$, and then using monotone class arguments (see Revuz and Yor \cite{key-18}, Exercise VI.(1.15), or Rogers and Williams \cite{key-19}, Theorem IV.(45.4)).
\end{proof}
Concerning continuity properties of local times, there is the following result.
\begin{lem}
\label{lem:Let--be}Let $X=\left(X_{t},\,t\geq0\right)$ be a continuous semimartingale with canonical decomposition given by $X=M+A$, where $M$ is a local martingale and $A$ a finite variation process. Then there exists a modification of the local time process $\left(L^X\left(t,x\right),t\geq0,\,x\in\mathbb{R}\right)$
of $X$ such that the map 
$\left(t,x\right)\mapsto L^X\left(t,x\right)$
is continuous in $t$ and \cadlag in $x$, $\mathbf{P}$-a.s.
Moreover
\beq
L^X\left(t,x\right)-L^X\left(t,x-\right) =2\int_{0}^{t}\mathbb{I}_{\{x\}}(X_s)\, \rmd A_{s},\label{eq:ii}
\eeq
for all $t\geq0,\,x\in\mathbb{R}$, $\mathbf{P}$-a.s.
\end{lem}
\begin{proof}
See, e.g., \cite{key-18}, Theorem VI.(1.7). 
\end{proof}
The information process $\bt$ is a continuous
semimartingale (cf. Theorem \ref{thm:beta semimartingaale}), hence the local time $L^\bt\left(t,x\right)$ of $\beta$ at level
$x\in\mathbb{R}$ up to time $t\geq0$ is well defined. The occupation time formula takes the following form. 
\begin{cor}
\label{lem:occupation-time2} We have 
\[
\intop_{0}^{t\wedge\tau}h\left(s,\beta_{s}\right)ds=\intop_{0}^{t}h\left(s,\beta_{s}\right)d\left\langle \beta,\beta\right\rangle _{s}=\intop_{-\infty}^{+\infty}\left(\intop_{0}^{t}h\left(s,x\right)\,\rmd L^\bt\left(s,x\right)\right)\,\rmd x\,,
\]
for all $t\geq0$ and all non-negative Borel functions $h$ on $\mathbb{R}_{+}\times\mathbb{R}$, $\mbf P$-a.s.
\end{cor}
\begin{proof}
The first equality follows from relation (\ref{eq: beta beta t}) and the second
is an application of Lemma \ref{lem:OccTimeFormula}.
\end{proof}
An important property of the local time $L^\bt$ is the existence of a bicontinuous version.
\begin{lem}
\label{lem:local time x continuity} There is a version of $L^\bt$ such that the map 
$(t,x)\in\mb R_+\times \mb R\mapsto L^\bt\left(t,x\right)$ is continuous, $\mathbf{P}$-a.s.
\end{lem}
\begin{proof}
We choose a version of the local time $L^\bt$ according to Lemma \ref{lem:Let--be}. Using (\ref{eq:ii}) we have that
\[
L^\bt\left(t,x\right)-L^\bt\left(t,x-\right)=-2\int_{0}^{t\wedge\tau}\mathbb{I}_{\{x\}}(\bt_s)\,u\left(s,\beta_{s}\right)ds,
\]
for all $t\geq0,\,x\in\mathbb{R}$, $\mathbf{P}$-a.s., where $u$
is the function defined by (\ref{eq:u}). Applying Corollary \ref{lem:occupation-time2} to the right-hand side of the last equality above, we see that
\[
2\intop_{0}^{t\wedge\tau}\mathbb{I}_{\{x\}}(\bt_s)\,u\left(s,\beta_{s}\right)\,\rmd s=2\intop_{-\infty}^{+\infty}\mathbb{I}_{\{x\}}(y)\left(\intop_{0}^{t}u\left(s,y\right)\,\rmd L^\bt\left(s,y\right)\right)\,\rmd y=0,
\]
and hence $L^\bt\left(t,x\right)-L^\bt\left(t,x-\right)=0$,
for all $t\geq0,\,x\in\mathbb{R}$, $\mathbf{P}$-a.s., because $\{x\}$ has Lebesgue measure zero. This completes the proof.
\end{proof}
We also make use of the boundedness of the local time with
respect to the space variable.
\begin{lem}
\label{lem:The-function-l is bounded}  The function $x\mapsto L^\bt\left(t,x\right)$ is bounded for all $t\in\mathbb{R}_{+}$ $\mbf P$-a.s. (the bound may depend on $t$ and $\omega$).
\end{lem}
\begin{proof}
It follows from the occupation time formula (or from Revuz and Yor \cite{key-18}, Corollary VI.(1.9)) that the local time $L^\bt(t,\cdot)$ vanishes outside of the compact interval $[-M_t(\om),M_t(\om)]$ where 
\beq
M_{t}\left(\omega\right) :=\sup_{s\in\left[0,t\right]}\left|\beta_{s}\left(\omega\right)\right|,\quad t\geq0,\;\omega\in\Omega\,,\label{eq:runningMAX}
\eeq
which together with the continuity of $L^\bt\left(t,\cdot\right)$ (see Lemma \ref{lem:local time x continuity}) yields the boundedness of this function, $\mbf P$-a.s.
\end{proof}
Outside a negligible set, for fixed $x\in\mb R$, the local time $L^\bt\left(\cdot,x\right)$
is a positive continuous increasing function, and we can associate
with it a random measure on $\mathbb{R}_{+}$:
\[
L^\bt\left(B,x\right):=\int_{B}\rmd L^\bt\left(s,x\right),\quad B\in\mathcal{B}\left(\mathbb{R}_{+}\right).
\]

\begin{lem}
\label{lem:weak convergence of loc time meas} Outside a negligible set, for any sequence $\left(x_{n}\right)_{n\in\mathbb{N}}$
in $\mathbb{R}$ converging to $x\in\mathbb{R}$, the sequence $\left(L^\bt\left(\cdot,x_{n}\right)\right)_{n\in\mathbb{N}}$
converges weakly to $L^\bt\left(\cdot,x\right)$, i.e.,
\[
\int_{\mb R_+} g\left(s\right)L^\bt\left(ds,x_{n}\right)\xrightarrow[n\rightarrow \infty]{}\int_{\mb R_+} g\left(s\right)L^\bt\left(ds,x\right),
\]
for all bounded and continuous functions $g:\mathbb{R}_{+}\mapsto\mathbb{R}$.
\end{lem}
\begin{proof}
We fix a negligible set outside of which $L^\bt$ is bicontinuous (cf. Lemma \ref{lem:local time x continuity}) and outside of which we will be working now. The measures $\left(L^\bt\left(\cdot,x_{n}\right)\right)_{n\in\mathbb{N}}$
are finite on $\mathbb{R}$ and they are supported by $\left[0,\tau\right]$.
By continuity of $L^\bt\left(t,\cdot\right)$ we have that
$L^\bt\left(s,x_{n}\right)\xrightarrow[n\rightarrow \infty]{}L^\bt\left(s,x\right),\,s\geq0$,
from which it follows that
\begin{equation}
L^\bt\left(\left[0,s\right],x_{n}\right)\xrightarrow[n\rightarrow \infty]{}L^\bt\left(\left[0,s\right],x\right),\quad s\geq0\,.\label{eq:star-1-1}
\end{equation}
We also have this convergence for the whole space $\mb R_+$:
\[
L^\bt\left(\mb R_+,x_{n}\right)=L^\bt\left(\left[0,\tau\right],x_{n}\right)\xrightarrow[n\rightarrow\infty]{}L^\bt\left(\left[0,\tau\right],x\right)=L^\bt\left(\mb R_+,x\right)\,.
\]
From this we can conclude that the measures $L^\bt\left(\cdot,x_{n}\right)$ converge weakly to $L^\bt\left(\cdot,x\right)$.
\end{proof}

\section{\label{sec:Auxiliary-Results}Auxiliary Results}
\noindent In \eqref{density-q} we had introduced the function $q$ by
$$
q(h,x):=\frac{1}{h}\int_0^h p(u,x,0)\,\rmd u,\quad 0<h\le 1, \ x\in\mb R\,,
$$
where $p(t,\cdot,y)$ is the density of the normal distribution with variance $t$ and expectation $y$ (see \eqref{eq:GaussDens}). 
\bg{lem}\label{q-weak convergence} The functions $q(h,\cdot)$ are probability density functions with respect to the Lebesgue measure on $\mb R$. The probability measures $\mb Q_h$ on $\mb R$ associated with the density $q_h$ converge weakly as $h\downarrow0$ to the Dirac measure $\delta_0$ at $0$.
\e{lem}
\bg{proof} The first statement of the lemma is obvious. For verifying the second statement, let $f$ be a bounded continuous function on $\mb R$. Using Fubini's theorem, we obtain
\beqas
\int_\mb R f(x)\,\mb Q_h(\rmd x)&=&\int_\mb R f(x)\,q_h(x)\,\rmd x\\
&=&\int_\mb R f(x)\,\Big(\frac{1}{h}\int_0^hp(u,x,0)\,\rmd u\Big)\,\rmd x\\
&=&\frac{1}{h}\int_0^h\Big(\int_\mb R f(x)\,p(u,x,0)\,\rmd x\Big)\,\rmd u\\
&=&\frac{1}{h}\int_0^h\Big(\int_\mb R f(x)\,\mc N(0,u)(\rmd x)\Big)\,\rmd u\,.
\eeqas
Since the function $u\in[0,1]\mapsto\mc N(0,u)$ which associates to every $u\in[0,1]$ the centered Gaussian law $\mc N(0,u)$ is continuous with respect to weak convergence of probability measures (note that $\mc N(0,0)=\delta_0$), we observe that the function $u\in[0,1]\mapsto\int_\mb R f(x)\,\mc N(0,u)(\rmd x)$ is continuous. An application of the fundamental theorem of calculus yields that the right-hand side converges to $\int_\mb R f(x)\,\delta_0(\rmd x)$ as $h\downarrow0$ and hence
$$
\lim_{h\downarrow0}\int_\mb R f(x)\,\mb Q_h(\rmd x)=f(0)\,,
$$
proving the second statement of the lemma.
\end{proof}
Now we consider the function $g$ introduced in \eqref{eq:g}:
$$
g\left(s,x\right):=\Big(\int_s^\infty\varphi_{s}\left(v,x\right)f\left(v\right)\,\rmd v\Big)^{-1},\quad s>0,\,x\in\mathbb{R}\,.
$$
\begin{lem}\label{lem:Lemma3} {\rm (1)} For all $x\in\mathbb{R}$ and $0<t_0<t$, the function $g\left(\cdot,x\right): [t_0,t]\mapsto\mathbb{R}$
is bounded, i.e., there exists a real constant $C\left(t_0,t,x\right)$ such that
\[
\sup_{s\in[t_0,t]}g\left(s,x\right)\leq C(t_0,t,x)\,.
\]

{\rm (2)} For all $x\in\mathbb{R}$ and $0<t_0<t$, the function $g(\cdot,x): [t_0,t]\mapsto\mathbb{R}$
is continuous, i.e., for all $s_n,s\in[t_0,t]$ such that $s_n\rightarrow s$,
\[
\lim_{s_{n}\rightarrow s}g(s_{n},x)=g(s,x)\,.
\]

{\rm (3)} Let $\left(x_{n}\right)_{n\in\mathbb{N}}$ be a sequence converging monotonically to $x\in\mathbb{R}$. Then, for all $0<t_0<t$,
\[
\sup_{s\in\left[t_0,t\right]}\left|g(s,x_{n})-g(s,x)\right| \xrightarrow[n\rightarrow \infty]{}0\,.
\]
\end{lem}
\begin{proof}
Let us define, for every $s\in\left[t_0,t\right]$ and $x\in\mathbb{R}$,
\beqa
\nonumber D\left(s,x\right)&&\\
&:=&\int_s^\infty \sqrt{\frac{v}{2\pi s\,(v-s)}}\exp\left(-\frac{v\,x^{2}}{2s\,(v-s)}\right)f\left(v\right)\,\rmd v\,,\label{eq:capital D}
\eeqa
and rewrite $g$ as 
\begin{equation}
g(s,x)=\frac{1}{D\left(s,x\right)},\quad s\in\left[t_0,t\right],\;x\in\mathbb{R}\,.\label{eq:implicit D}
\end{equation}
In order to prove statement (1), it suffices to verify that there
exists a constant $\tilde{C}\left(t_0,t,x\right)$ such that
\begin{equation}
0<\tilde{C}\left(t_0,t,x\right)\leq D\left(s,x\right),\quad s\in[t_0,t],\; x\in\mb R\,.\label{eq:minorizing}
\end{equation}
Such a constant can be found by setting
\begin{equation}
\tilde{C}\left(t_0,t,x\right):=\int_t^\infty\sqrt{\frac{1}{2\pi t}}\exp\left(-\frac{v\,x^{2}}{2t_0(v-t)}\right)f(v)\,\rmd v\,,\label{eq: C tilde}
\end{equation}
proving the first statement of the lemma. 

In order to prove statement (2) of the lemma, it suffices to
verify that the function $s\mapsto D\left(s,x\right),\,s\in[t_0,t]$,
is continuous, a fact that can be proved using Lebesgue's dominated convergence theorem. Indeed, let $s_n,s\in[t_0,t]$ such that $s_n\rightarrow s$ as $n\rightarrow\infty$. Rewriting \eqref{eq:capital D} we get
\beqas
\lefteqn{D\left(s_n,x\right)}\\
&=&\int_{t_0}^\infty \mb I_{(s_n,+\infty)}(v)\sqrt{\frac{v}{2\pi s_n\,(v-s_n)}}\exp\left(-\frac{v\,x^{2}}{2s_n\,(v-s_n)}\right)f\left(v\right)\,\rmd v\,.
\eeqas
First we consider the integral from $t$ to $\infty$: For $v\geq t$, we can make an upper estimate of the integrand by $\sqrt{\frac{v}{2\pi t_0\,(v-t)}}\, f(v)$ which is integrable over $[t,+\infty)$. For the second part of the integral from $t_0$ to $t$ we estimate the integrand by $\mb I_{(s_n,+\infty)}(v)\sqrt{\frac{t}{2\pi t_0\,(v-s_n)}}\,c$, where $c$ is an upper bound of $f$ on $[t_0,t]$, and by integrating we observe that  
$$
\lim_{n\rightarrow\infty}\int_{t_0}^t\!\mb I_{(s_n,+\infty)}(v) \sqrt{\frac{t}{2\pi t_0\,(v-s_n)}}\,\rmd v
\!=\!\int_{t_0}^t\!\mb I_{(s,+\infty)}(v)\sqrt{\frac{t}{2\pi t_0\,(v-s)}}\,\rmd v\,.
$$
As the integrands are nonnegative, we get convergence in $L^1([t_0,t])$ and hence uniform integrability (cf. Theorem \ref{thm:Lebesgue2}). This means that the sequence 
$$
I_{(s_n,+\infty)}(v)\sqrt{\frac{v}{2\pi s_n\,(v-s_n)}}\exp\left(-\frac{v\,x^{2}}{2s_n\,(v-s_n)}\right)f\left(v\right)
$$
is uniformly integrable on $[t_0,t]$ and we can apply Lebesgue's theorem (cf. Theorem \ref{thm:Lebesgue2}) to conclude 
\beqas
\lefteqn{\lim_{n\rightarrow\infty}\int_{t_0}^t \mb I_{(s_n,+\infty)}(v)\sqrt{\frac{v}{2\pi s_n\,(v-s_n)}}\exp\left(-\frac{v\,x^{2}}{2s_n\,(v-s_n)}\right)f\left(v\right)\,\rmd v}\\
&=&\int_{t_0}^t \mb I_{(s,+\infty)}(v)\sqrt{\frac{v}{2\pi s\,(v-s)}}\exp\left(-\frac{v\,x^{2}}{2s\,(v-s)}\right)f\left(v\right)\,\rmd v\,.
\eeqas
Summarizing we get 
$$
\lim_{n\rightarrow\infty}D\left(s_n,x\right)=D\left(s,x\right)
$$
and the proof of statement (2) of the lemma finished.

We turn to the proof of statement (3) of the lemma. Using relation (\ref{eq:implicit D}) we see that
\[
\left|g\left(s,x_{n}\right)-g\left(s,x\right)\right| =\frac{\left|D\left(s,x_{n}\right)-D\left(s,x\right)\right|}
{D\left(s,x_{n}\right)D\left(s,x\right)}
\]
and from inequality (\ref{eq:minorizing}) we get that 
\[
\sup_{s\in\left[t_0,t\right]}\left|g\left(s,x_{n}\right)-g\left(s,x\right)\right|\leq\frac{\sup_{s\in\left[t_0,t\right]}\left|D\left(s,x_{n}\right)-D\left(s,x\right)\right|}{\tilde{C}\left(t_0,t,x_{n}\right)\,\tilde{C}\left(t_0,t,x\right)},
\]
where $\tilde{C}\left(t_0,t,x\right)$ is defined by (\ref{eq: C tilde}). It is easy to see that 
\[
\lim_{n\rightarrow \infty}
\frac{1} {\tilde{C}\left(t_0,t,x_{n}\right)\tilde{C}\left(t_0,t,x\right)}=\frac{1}{\tilde{C}\left(t_0,t,x\right)^{2}}<+\infty\,.
\]
Hence it remains to prove that
\[
\sup_{s\in\left[0,t\right]}\left|D\left(s,x_{n}\right)-D\left(s,x\right)\right|\xrightarrow[n\rightarrow \infty]{}0.
\]
By assumption, the sequence $x_{n}$ converges monotonically to $x$. In such a case it is easy to see that the sequence of functions $D\left(\cdot,x_{n}\right)$ is monotone. Furthermore, using Lebesgue's dominated convergence theorem, we verify that $D\left(s,x_{n}\right)$ converges to $D\left(s,x\right)$, for all $s\in\left[t_0,t\right]$. Since the function $s\mapsto D\left(s,x\right)$ is
also continuous on $[t_0,t]$, according to Dini's theorem, $D\left(\cdot,x_{n}\right)$ converges uniformly to $D\left(\cdot,x\right)$ on $\left[t_0,t\right]$. This
implies the third statement of the lemma and the proof is finished.
\end{proof}
\begin{lem}
\label{lem:Let--be-AUX}Let $h,h_{n}$
be bounded and continuous functions on a metric space $E$ and $\mu,\mu_{n}$ be finite measures on $\left(E,\mathcal{B}\left(E\right)\right)$.
Suppose that the following two conditions are satisfied:
\begin{enumerate}
\item The sequence of functions $h_{n}$ converges uniformly to $h$.
\item The sequence of measures $\mu_{n}$ converges weakly to $\mu$.
\end{enumerate}
Then $\lim_{n\uparrow+\infty}\int_{E}h_{n}\,\rmd\mu_{n}=\int_{E}h \,\rmd\mu$.
\end{lem}
\begin{proof}
It can immediately be verified that 
\begin{multline*}
\Big|\int_{E}h_{n}\,\rmd\mu_{n}-\int_{E}h \,\rmd\mu\Big|\\
\leq\sup_{x\in E}\Big|h\left(x\right)-h_{n}\left(x\right)\Big|\int_{E}\,\rmd \mu_{n}+\Big|\int_{E}h \,\rmd \mu_{n}-\int_{E}h\, \rmd \mu\Big|,
\end{multline*}
which converges to 0 as $n\uparrow+\infty$.
\end{proof}
\begin{lem}
\label{lem:Lemma4} Let $0<t_0<t$. The function $k:\,\mb R\rightarrow\mathbb{R}_{+}$ given by
\[
k\left(x\right):=\int_{t_0}^{t}g\left(s,x\right)\,f\left(s\right)\,\rmd L^\bt(s,x),\quad x\in \mb R\,,
\]
is bounded and continuous, where the function $g$ is given by (\ref{eq:g}). 
\end{lem}
\begin{proof} Let us first restrict to a compact subset $E$ of $\mb R$. First we prove the right and left continuity, hence the continuity, of the function $k$. Let $x_{n}$ be a sequence from $E$ converging monotonically to $x\in E$. From Lemma \ref{lem:Lemma3}
we know that the bounded and continuous functions $g\left(\cdot,x_{n}\right):\,[t_0,t]\rightarrow\mb R$
converge uniformly to the bounded and continuous function $g\left(\cdot,x\right):\,[t_0,t]\rightarrow\mb R$ as $n\rightarrow\infty$. From Lemma \ref{lem:weak convergence of loc time meas},
we obtain that the sequence of measures $L^\bt\left(\cdot,x_{n}\right)$
converges weakly to $L^\bt\left(\cdot,x\right)$ as $n\rightarrow\infty$. Applying Lemma \ref{lem:Let--be-AUX}, we have that
\beqas
\lim_{n\rightarrow\infty}k\left(x_{n}\right)&=&\lim_{n\rightarrow\infty}\int_{t_0}^{t}g\left(s,x_{n}\right)f(s)\,\rmd L^\bt\left(s,x_{n}\right)\\
&=&\int_{t_0}^{t}g\left(s,x\right)f(s)\,\rmd L^\bt\left(s,x\right)=k\left(x\right)\,.
\eeqas
Consequently, the function $k$ is continuous on $E$. The boundedness of $k$ now follows from the compactness of $E$. In order to show that the statement also holds for $\mb R$, let us choose $E=[-M_t-1,M_t+1]$ (see \eqref{eq:runningMAX} for notation). As $L^\bt\left(s,x\right)=0$, $s\in[0,t],\,x\notin[-M_t,M_t]$ (see the proof of Lemma \ref{lem:The-function-l is bounded}), the statement follows. 
\end{proof}

\section{\label{sec:The-Meyer-Approach}The Meyer Approach to the Compensator}

\noindent Below we briefly recall the approach developed by P.-A. Meyer
\cite{key-17} for computing the compensator of a right-continuous potential of class (D).
In this section $\mathbb{F}=\left(\mathcal{F}_{t}\right)_{t\geq0}$
denotes a filtration satisfying the usual hypothesis of right-continuity and completeness.

We begin with the definition of a right-continuous potential of class (D). Let $X=\left(X_{t},\,t\geq0\right)$ be a right-continuous $\mathbb{F}$-super\-mar\-tin\-gale and let $\mathcal{T}$ be the collection of all finite $\mathbb{F}$-stopping times relative to this family. The process $X$ is said to \textit{belong to the class }(D) if the collection of random variables
$X_{T},\,T\in\mathcal{T}$, is uniformly integrable. We say that the right-continuous supermartingale $X$ is a \textit{potential} if the random variables $X_{t}$ are non-negative and if 
\[
\lim_{t\rightarrow+\infty}\mathbf{E}\left[X_{t}\right]=0.
\]
\begin{defn}
\label{def:potential generated by A}Let $C=\left(C_{t},\,t\geq0\right)$
be an integrable $\mathbb{F}$-adapted right-continuous increasing
process, and let $L=\left(L_{t},\,t\geq0\right)$ be a right-continuous modification of the martingale $\left(\mathbf{E}\left[C_{\infty}|\mathcal{F}_{t}\right],\,t\geq0\right)$; the process $Y=\left(Y_{t},\,t\geq0\right)$ given by 
\[
Y_{t}:=L_{t}-C_{t}
\]
is called the \textit{potential generated by }$C$.
\end{defn}
The following result establishes a connection between potentials generated by an increasing process and potentials of class (D). Let $h$ be a strictly positive real number and $X=\left(X_{t},\,t\geq0\right)$
be a potential of class (D), and denote by $\left(p_{h}X_{t},\,t\geq0\right)$
the right-continuous modification of the supermartingale $\left(\mathbf{E}\left[X_{t+h}|\mathcal{F}_{t}\right],\,t\geq0\right)$.
\begin{thm}
\label{thm:prop of Xh}Let $X=\left(X_{t},\,t\geq0\right)$ be a potential of class (D), let $h>0$ and $A^{h}=\left(A_{t}^{h},\,t\geq0\right)$
be the process defined by
\begin{equation}
A_{t}^{h}:=\frac{1}{h}\intop_{0}^{t}\left(X_{s}-p_{h}X_{s}\right)ds.\label{eq:A acca}
\end{equation}
Then $A^{h}$ is an integrable increasing process which generates a potential of class (D) $X^{h}=\left(X_{t}^{h},\,t\geq0\right)$
dominated by $X$, i.e., the process $X-X^{h}$ is a potential. It holds
$$
X_t^h=\frac{1}{h}\mbf E\Big[\int_0^hX_{t+s}\,\rmd s|\mc F_t\Big], \quad \mbf P\mbox{-a.s.},\; t\geq 0\,.
$$
\end{thm}
\begin{proof}
See, e.g., \cite{key-17}, VII.T28.
\end{proof}
An increasing process $A=\left(A_{t},\,t\geq0\right)$ is called \textit{natural} (with respect to the filtration $\mb F$)
if, for every bounded right-continuous $\mb F$-martingale $M=\left(M_{t},\,t\geq0\right)$,
we have
\[
\mathbf{E}\Big[\intop_{\left(0,t\right]}M_{s}\,\rmd A_{s}\Big]=\mathbf{E}\Big[\intop_{\left(0,t\right]}M_{s-}\,\rmd A_{s}\Big],\quad t>0\,.
\]

It is well known that an increasing process $A$ is natural with respect to $\mb F$ if and only if it is $\mb F$-predictable.

For the following definition of convergence in the sense of the weak topology $\sigma\left(L^{1},L^{\infty}\right)$, see \cite{key-17},
II.10.
\begin{defn}
\label{def:weak L1 Linf conv} Let $\left(\xi_{n}\right)_{n\in\mathbb{N}}$ be a sequence of integrable real-valued random variables. The sequence
$\left(\xi_{n}\right)_{n\in\mathbb{N}}$ is said to \textit{converge
to an integrable random variable $\xi$ in the weak topology $\sigma\left(L^{1},L^{\infty}\right)$} if 
\[
\lim_{n\rightarrow+\infty}\mathbf{E}\left[\xi_{n}\eta\right]=\mathbf{E}\left[\xi\eta\right],\;\textrm{for all }\eta\in L^{\infty}\left(\mathbf{P}\right).
\]
\end{defn}
\begin{thm}
\label{thm:prop of Ah}Let $X=\left(X_{t},\,t\geq0\right)$ be a right-continuous potential of class (D). Then there exists an integrable natural increasing process $A=\left(A_{t},\,t\geq0\right)$ which generates
$X$, and this process is unique. For every stopping time $T$ we
have
\[
A_{T}^{h}\xrightarrow[h\downarrow0]{^{\sigma\left(L^{1},L^{\infty}\right)}}A_{T}.
\]
\end{thm}
\begin{proof}
See, e.g., \cite{key-17}, VII.T29.
\end{proof}
In the framework of the information based approach, the process $H=\left(H_{t},\,t\geq0\right)$
given by (\ref{eq:H}) is a bounded increasing process
which is $\mathbb{F}^{\beta}$-adapted. It is a submartingale and
it can be immediately seen that the process $G=\left(G_{t},\,t\geq0\right)$
given by (\ref{eq:G})  is a right-continuous potential of class (D).
By Theorem \ref{thm:prop of Xh}, the processes $K^{h},\,h>0$, defined
by (\ref{eq:KAPPAacca-1}), generate a family of potentials $G^{h}$
dominated by $G$.
\begin{cor}
\label{thm:Kh Kw} There exists a unique integrable natural increasing
process $K^{w}=\left(K_{t}^{w},\,t\geq0\right)$ which generates the
process $G$ defined by (\ref{eq:G}) and, for every $\mathbb{F}^{\beta}$-stopping
time $T$, we have that 
\[
K_{T}^{h}\xrightarrow[h\downarrow0]{^{\sigma\left(L^{1},L^{\infty}\right)}}K_{T}^{w},
\]
 where $K^{h}$ is the process defined by (\ref{eq:KAPPAacca-1}).\end{cor}
\begin{proof}
See Theorem \ref{thm:prop of Ah}.
\end{proof}
\begin{thm}[Compactness Criterion of Dunford-Pettis]
\label{thm:(Compactness-Criterion-of}
Let $\mathcal{A}$ be a subset of the space $L^{1}\left(\mathbf{P}\right)$.
The following two properties are equivalent:
\begin{enumerate}
\item $\mathcal{A}$ is uniformly integrable;
\item $\mathcal{A}$ is relatively compact in the weak topology $\sigma\left(L^{1},L^{\infty}\right)$.
\end{enumerate}
\end{thm}
\begin{proof}
See \cite{key-17}, II.T23.
\end{proof}
\begin{thm}
\label{thm:Lebesgue2}Let $\left(\xi_{n}\right)_{n\in\mathbb{N}}$
be a sequence of integrable random variables converging in probability to a random variable $\xi$. Then $\xi_{n}$ converges to $\xi$ in $L^{1}(\mbf P)$ if and only if $\left(\xi_{n}\right)_{n\in\mathbb{N}}$ is uniformly integrable. If the random variables $\xi_{n},\,n\geq1,$ are non-negative they are uniformly integrable if and only if
\[
\lim_{n\rightarrow+\infty}\mathbf{E}\left[\xi_{n}\right]=\mathbf{E}\left[\xi\right]<+\infty\,.
\]
\end{thm}
\begin{proof}
See \cite{key-17}, II.T21.
\end{proof}
\begin{lem}
\label{cor:Let--be-1}Let $\left(\xi_{n}\right)_{n\in\mathbb{N}}$
be a sequence of random variables and $\xi,\eta\in L^{1}\left(\mathbf{P}\right)$
such that:
\begin{enumerate}
\item $\xi_{n}\xrightarrow[\:n\rightarrow+\infty]{\sigma\left(L^{1},L^{\infty}\right)}\eta;$
\item $\xi_{n}\rightarrow\xi, \;\mathbf{P}$-a.s.
\end{enumerate}
Then $\eta=\xi,\;\mathbf{P}\textrm{-a.s.}$
\end{lem}
\begin{proof}
From condition (1) we see that $\left(\xi_{n}\right)_{n\in\mathbb{N}}$
is relatively compact in the weak-topology $\sigma\left(L^{1},L^{\infty}\right)$.
By Theorem \ref{thm:(Compactness-Criterion-of} it follows that the
family $\left(\xi_{n}\right)_{n\in\mathbb{N}}$ is uniformly integrable. We also know that $\xi_{n}\rightarrow\xi$ $\mathbf{P}$-a.s. Hence,
by Theorem \ref{thm:Lebesgue2}, we see that $\xi_{n}\rightarrow\xi$
in the $L^{1}$-norm and, consequently, $\xi_{n}\xrightarrow[\:n\rightarrow+\infty]{\sigma\left(L^{1},L^{\infty}\right)}\xi$.
The statement of the lemma then follows by the uniqueness of the limit.
\end{proof}
\end{appendix}
\subsection*{Acknowledgment}
\noindent This work has been financially supported by the European Community's FP 7 Program under contract PITN-GA-2008-213841, Marie Curie ITN \flqq Controlled Systems\frqq.

\end{document}